\documentclass{amsart}

\usepackage{iftex}
\ifPDFTeX\usepackage[utf8]{inputenc}\fi
\usepackage[T1]{fontenc}
\usepackage{amsfonts}
\usepackage{amsmath}
\usepackage{amsthm}
\usepackage{amssymb}
\usepackage{latexsym}
\usepackage{enumerate}
\usepackage{comment}
\usepackage{hyperref}
\usepackage{xcolor}

\usepackage{tikz}
\usetikzlibrary{graphs, positioning}
\usepackage[square,numbers]{natbib}

\newtheorem{theorem}{Theorem}[section]
\newtheorem*{theorem*}{Theorem}
\newtheorem{corollary}[theorem]{Corollary}
\newtheorem*{corollary*}{Corollary} 
\newtheorem{definition}[theorem]{Definition}
\newtheorem{example}[theorem]{Example}
\newtheorem{lemma}[theorem]{Lemma}

\newtheorem{proposition}[theorem]{Proposition}

\numberwithin{equation}{section}

\newtheorem{question}[theorem]{Question}

\newcommand{\K}[1]{\mathcal{K}({#1})}
\newcommand{\B}[2]{\overline{B}({#1, #2})}
\newcommand{\Sp}[2]{S({#1, #2})}
\newcommand{\set}[1]{\{#1\}}

\newcommand{\NN}{\mathbb{N}}
\newcommand{\R}{\mathbb{R}}
\newcommand{\Z}{\mathbb{Z}}
\newcommand{\Td}{T_{k}}
\newcommand{\dH}{d_H}

\title{Plasticity in graph metric spaces and their hyperspaces}
\author{Clayton Suguio Hida $^{\dagger, \ast}$}
\address{$^{\dagger}$Universidade do Estado do Amap\'a - UEAP, Brazil}
\address{$^{\ast}$ Corresponding author: Clayton Suguio Hida. email: clayton.hida@ueap.edu.br}

\begin{document}

\begin{abstract}
A metric space is plastic if every bijective nonexpansive self-map is an isometry.
We prove that the hyperspace of nonempty compact subsets of every connected,
locally finite, regular graph, equipped with the Hausdorff metric associated with
the path metric, is plastic. The proof identifies the sets with smallest unit
balls as the nonempty subsets of adjacent-twin classes and shows that every
nonexpansive bijection induces an automorphism of the quotient graph preserving
the sizes of these classes. Singletons are preserved when there are no adjacent
twins, but need not be preserved in general. We also prove that
$\K{K\times G}$ is plastic for every compact connected metric space $K$ and every
connected, locally finite, regular graph $G$, with the supremum metric on the
product, and give a more general criterion for such products. Further results
include a plastic metric space whose hyperspace is not plastic, plasticity of
every tree, and plasticity of the hyperspace of a connected, locally finite graph
with only finitely many vertices of minimum degree.

\smallskip

\end{abstract}

\maketitle

\tableofcontents

\section{Introduction}\label{sec:intro}

A map $f\colon X\to Y$ between metric spaces is \emph{nonexpansive} if
$d(f(x),f(y))\le d(x,y)$ for all $x,y\in X$. A metric space $X$ is
called \emph{plastic} if every bijective nonexpansive map
$f\colon X\to X$ is an isometry. A classical theorem of Freudenthal and
Hurewicz \cite{freudenthal1936} implies that every totally bounded metric
space is plastic. The terminology and the systematic study of plasticity
were introduced by Naimpally, Piotrowski and Wingler
\cite{naimpally2006plasticity}.

Beyond totally bounded spaces, plasticity is considerably more delicate.
Much of the recent work has concerned the unit balls of Banach spaces.
Cascales, Kadets, Orihuela and Wingler \cite{CKOW} proved that the unit ball
of every strictly convex Banach space is plastic. Whether the unit ball of
every Banach space is plastic remains open, although affirmative answers are
known for several classes, including $\ell_1$-sums of strictly convex spaces
\cite{KZell1}, $c$ and $c_0$ \cite{Leo}, certain $C(K)$ spaces
\cite{Fakhoury}, and the real space $\ell_\infty$ \cite{HallerLeo2026}.
Conditional results for $\ell_\infty$-sums appear in \cite{Kurik}, and
further results can be found in \cite{Zavarzina2018,HLZ,Karpenko,AKZ}.
Plasticity has also been studied for subsets of the real line
\cite{LangemannZavarzina,Bielas} and for metric groups \cite{BMZ}. Kadets
and Zavarzina \cite{KZpairs} introduced the related notion of a
\emph{plastic pair}.

This paper concerns plasticity of hyperspaces. For a metric space $X$, let
$\K X$ denote the family of all nonempty compact subsets of $X$, equipped
with the Hausdorff metric \cite{michael1951,nadler1978hyperspaces}. In
\cite{hida2026plasticity}, the author proved that plasticity of $\K X$
implies plasticity of $X$ and asked whether the converse holds.

\begin{question}\label{main-question}
Is $\K X$ plastic whenever $X$ is plastic?
\end{question}

The answer is negative. The counterexample, suggested by G.~Vitturi during
the XXVI Brazilian Topology Meeting, is a countable uniformly discrete
metric space of diameter $2$, obtained by replacing each natural number by
a pair of points at distance $1$. This space is plastic
(Proposition~\ref{prop:Xm-plastic}), but its hyperspace is not
(Theorem~\ref{counterexample}). The obstruction lies in $\K X$, which
splits into clusters of different finite sizes: a bijection can map several
of these clusters into one, strictly decreasing some distances without
increasing any.

The counterexample shows that additional assumptions are needed for
plasticity to pass from a space to its hyperspace. It also exhibits a local
symmetry: the two points in each pair have the same closed unit ball. In
graphs, adjacent vertices with this property are called \emph{adjacent
twins}, and understanding how such symmetries act on the hyperspace is
central to our approach.

We focus on connected graphs endowed with the unweighted path metric. A
basic positive example is $\mathbb Z$, viewed as a connected $2$-regular
graph: both $\mathbb Z$ and its hyperspace are plastic, the latter by
\cite[Theorem 3.8]{hida2026plasticity}. This leads to the following
question.

\begin{question}\label{graph-question}
Let $G$ be a plastic connected graph, endowed with the unweighted path
metric. Is $\K G$ plastic?
\end{question}

Connectedness and local finiteness alone do not guarantee plasticity of
$G$ (Example~\ref{ex:nonplastic}). Adding regularity does: every connected,
locally finite, regular graph is plastic
(Theorem~\ref{Plasticity-regular-graph}). Our main theorem establishes the
corresponding result for its hyperspace.

\begin{theorem*}[Theorem~\ref{every-regular-graph-plastic}]
Let $G$ be a connected, locally finite, regular graph. Then $\K G$ is
plastic.
\end{theorem*}

The main ingredient in the proof is a characterization of the sets whose
unit balls in the hyperspace have minimum cardinality. For every connected,
locally finite graph of minimum degree $k\ge1$, we prove that
\[
\min_{A\in\K G}|\B A1|=2^{k+1}-1.
\]
Moreover, equality holds exactly when $A$ is a nonempty subset of a class
of adjacent twins whose vertices have minimum degree
(Theorem~\ref{thm:min-hyperballs}). The proof combines estimates based on
dominating sets with a count of the subsets of the closed neighbourhood of
$A$. This characterization does not require regularity and will also be
used in the nonregular case.

For a regular graph without adjacent twins, the sets attaining this minimum
are precisely the singletons. In the presence of twins, their role is
played by the finite families $\K C$, where $C$ ranges over the twin
classes. We show that every nonexpansive bijection of $\K G$ permutes these
families and induces an automorphism of the quotient graph that preserves
class sizes (Theorem~\ref{automorphism_g}). Lifting this automorphism to an
isometry of $G$ and composing with the induced hyperspace map reduces the
proof to a map that preserves each family $\K C$. We then show that such a
map preserves the hyperspaces of finite balls. Its restriction to each of
these finite metric spaces is an isometry, and these restrictions together
yield an isometry of $\K G$.

This argument must account for the fact that hyperspace maps need not
preserve singletons. For regular trees, every nonexpansive bijection of the
hyperspace does preserve them (Corollary~\ref{main-regular-trees}). By
contrast, adjacent twins allow a hyperspace isometry to exchange a singleton
with a two-point set (Lemma~\ref{twins}). Thus, even in the regular case, a
hyperspace isometry need not be induced by an isometry of the underlying
graph.

The description of how nonexpansive bijections act on hyperspaces of finite
balls also allows us to extend \cite[Theorem 4.5]{hida2026plasticity} to
products. The following criterion isolates the property required of the
discrete factor. All products below carry the supremum metric.

\begin{theorem*}[Theorem~\ref{main_product}]
Let $(K,d_K)$ be a compact connected metric space and let $(L,d_L)$ satisfy
the following conditions:
\begin{enumerate}
\item $d_L(u,v)\ge1$ for all distinct $u,v\in L$, and every bounded subset
of $L$ is finite.
\item For every nonexpansive bijection $h$ of $\K L$, there exist a
bijective isometry $g\colon L\to L$ and nonempty finite sets
$L_1\subseteq L_2\subseteq\cdots$, with $\bigcup_nL_n=L$, such that
\[
h[\K{L_n}]=\K{g[L_n]}\qquad(n\ge1).
\]
\end{enumerate}
Then $\K{K\times L}$ is plastic.
\end{theorem*}

The connected components of $\K{K\times L}$ are indexed by the nonempty
finite subsets of $L$. Compactness and connectedness of $K$ allow us to
show that every nonexpansive bijection of $\K{K\times L}$ induces a
nonexpansive bijection of $\K L$. The hypothesis on $L$ then allows us to
compose with an isometry so that each compact hyperspace
$\K{K\times L_n}$ is preserved. Plasticity of these compact subspaces
yields plasticity of the whole hyperspace.

For a connected, locally finite, regular graph, the second condition holds
with $L_n$ equal to the ball of radius $n$ centered at a fixed vertex. We
therefore obtain the following consequence.

\begin{corollary*}[Corollary~\ref{cor:product-regular}]
Let $K$ be a compact connected metric space and let $G$ be an infinite
connected, locally finite, regular graph. Then $\K{K\times G}$ is plastic.
In particular, $\K{K\times T_k}$ is plastic for every regular tree $T_k$ of
degree $k\ge2$.
\end{corollary*}

This applies to regular graphs with or without adjacent twins and, by
\cite{hida2026plasticity}, also gives plasticity of $K\times G$ itself.

We also give a partial answer to Question~\ref{graph-question} beyond the
regular case, under a finiteness condition on the vertices of minimum
degree.

\begin{theorem*}[Corollary~\ref{cor:finite-min-degree}]
Let $G$ be an infinite connected, locally finite graph, $k$ its minimum
degree and $G'=\{v:\deg(v)=k\}$. If $G'$ is finite, then both $G$ and
$\K G$ are plastic.
\end{theorem*}

The proof uses a general criterion for metric spaces with finite balls:
plasticity follows if only finitely many points have unit balls of minimum
cardinality (Theorem~\ref{general_finite_minimum}). A nonexpansive bijection
permutes these points, and, following their finite orbits, we compare
successive balls using sums of pairwise distances and the criterion for
finite plastic pairs from \cite{KZpairs}. For finite graph metrics, these
sums are the Wiener index. Combined with
Theorem~\ref{thm:min-hyperballs}, this criterion proves the preceding
corollary; in particular, it applies to locally finite trees with a
nonempty finite set of leaves.

When there are infinitely many vertices of minimum degree, neither this
criterion nor the argument for regular graphs applies directly. Thus,
Question~\ref{graph-question} remains open for such locally finite graphs.

The paper is organized as follows. Section~\ref{sec:prelim} recalls the
necessary facts about plasticity and the Hausdorff metric and constructs the
counterexample to Question~\ref{main-question}. Section~\ref{sec:graphs}
introduces graph metric notation and proves plasticity of connected,
locally finite, regular graphs. It also explains the restriction to
unweighted path metrics: a weighted path need not be plastic, even though
its underlying graph is $2$-regular (Example~\ref{ex:weighted}).
Section~\ref{sec:regular} establishes the minimum-cardinality
characterization, proves the main theorem, and examines the role of twins
in singleton preservation (Corollary~\ref{cor:twins-main}).
Section~\ref{sec:extras} proves the product theorem and its consequences.
Finally, Section~\ref{sec:nonregular} develops the finite-minimum criterion
and applies it to nonregular graphs.

\section{Preliminaries}\label{sec:prelim}

\subsection{Plasticity and hyperspaces}

Let $(X,d)$ be a metric space. For $x\in X$ and $r>0$, we denote by $\B x r$ the
closed ball centred at $x$ with radius $r$, and by
$\Sp x r=\{y\in X: d(x,y)=r\}$ the corresponding sphere. We denote by $\K X$ the hyperspace of all
nonempty compact subsets of $X$, equipped with the Hausdorff metric
\[
d_H(A,B)=\max\Big\{\sup_{a\in A}d(a,B),\ \sup_{b\in B}d(b,A)\Big\}.
\]

The following elementary observation from
\cite{hida2026plasticity} will be used extensively throughout the paper.

\begin{lemma}\label{general_fact_injectivity}
Let $(X,d)$ be a metric space and let $f\colon X\to X$ be a nonexpansive injection.
Then for every $x\in X$ and $r>0$, we have $f[\B x r]\subset\B{f(x)}{r}$. In particular, $|\B x r|\le|\B{f(x)}{r}|$.
\end{lemma}

\begin{proof}
See \cite[Lemma 3.1]{hida2026plasticity}.
\end{proof}

\begin{theorem}[\cite{freudenthal1936,naimpally2006plasticity}]\label{tb-plastic}
Every totally bounded metric space is plastic. In particular, every finite metric
space is plastic.
\end{theorem}

\begin{proposition}\label{hyper-total-bound}
Let $X$ be a totally bounded metric space. Then $\K X$ is plastic.
\end{proposition}

\begin{proof}
See \cite[Proposition 2.5]{hida2026plasticity}.
\end{proof}

\begin{lemma}\label{lem:hyperball}
Let $X$ be a metric space, let $\K X$ be its hyperspace, and let $A\in\K X$ and $n>0$. Then
for $U\in\K X$,
\[
\dH(A,U)\le n
\quad\Longleftrightarrow\quad
U\subseteq\bigcup_{a\in A}\B{a}{n}
\ \ \text{ and }\ \
U\cap\B{a}{n}\ne\emptyset\ \text{ for every }a\in A .
\]
\end{lemma}

\begin{proof}
$\sup_{u\in U}d(u,A)\le n$ says exactly that every $u\in U$ is within $n$ of some
$a\in A$, i.e.\ $U\subseteq\bigcup_{a}\B{a}{n}$; and $\sup_{a\in A}d(a,U)\le n$
says that every $a\in A$ has some $u\in U$ with $d(a,u)\le n$, i.e.\
$U\cap\B{a}{n}\ne\emptyset$.
\end{proof}

\subsection{Plasticity is not preserved by the hyperspace
construction}\label{sec:hyperspace-fails}

In this section, we provide a counterexample to \cite[Question 2.4]{hida2026plasticity}.

\begin{definition}\label{def:Xm}
Let $X_2=\NN\times\{0,1\}$, and for $(n,i),(n',j)\in X_2$ put
\[
d_2\bigl((n,i),(n',j)\bigr)=
\begin{cases}
0,&(n,i)=(n',j),\\
1,&(n,i)\ne(n',j)\ \text{ and }\ n=n',\\
2,&n\ne n'.
\end{cases}
\]
We call the sets $C_n=\{n\}\times\{0,1\}$ the \emph{clusters} of $X_2$, and write
$\pi(x)=n$ when $x\in C_n$; for $A\in\K{X_2}$ we write
$\pi[A]=\{n\in\NN:A\cap C_n\ne\emptyset\}$.
\end{definition}

\begin{lemma}\label{lem:cluster-maps}
Let $f\colon X_2\to X_2$ be a bijection. Then $f$ is nonexpansive if and only if it
maps each cluster \emph{into} a cluster, and $f$ is an isometry if and only if it
maps each cluster \emph{onto} a cluster.
\end{lemma}

\begin{proof}
The proof follows directly from the definition of the metric $d_2$ and the fact that $(C_n)_n$ forms a partition of $X_2$.
\end{proof}

\begin{proposition}\label{prop:Xm-plastic}
The space $(X_2,d_2)$ is plastic.
\end{proposition}

\begin{proof}
Let $f:X_2\to X_2$ be a nonexpansive bijection. By Lemma~\ref{lem:cluster-maps}, $f$
maps each cluster $C_n$ into a cluster $C_{\sigma(n)}$. Since $|C_n|=|C_{\sigma(n)}|
=2$ and $f$ is injective, we conclude that $f[C_n]=C_{\sigma(n)}$. By Lemma~\ref{lem:cluster-maps} again, $f$ is an isometry.
\end{proof}

\begin{theorem}\label{counterexample}
The space $\K{X_2}$ is not plastic. 
\end{theorem}

\begin{proof}
For a finite nonempty set $F\subset\NN$, put $K_F=\{A\in\K{X_2}:\pi[A]=F\}$, so that
$\K{X_2}$ is the disjoint union of the sets $K_F$. The definition of $d_2$ gives the following description of the Hausdorff metric:

\[
d_H\bigl(A,B\bigr)=
\begin{cases}
0,&A=B,\\
1,& A\ne B\ \text{ and }\ \pi[A]=\pi[B],\\
2,& \pi[A]\ne\pi[B].
\end{cases}
\]

Thus, $\K{X_2}$ has the same metric pattern as $X_2$: it is partitioned into the clusters
$K_F$, with two distinct points in the same cluster at distance $1$ and points in
distinct clusters at distance $2$. In particular, a bijection of $\K{X_2}$ is
nonexpansive if and only if it maps each $K_F$ into some $K_{F'}$, and it is an
isometry if and only if it maps each $K_F$ onto some $K_{F'}$. Observe that, unlike in $X_2$, where all clusters have the same cardinality, the clusters $K_F$ may have different cardinalities. This variation in cardinality allows us to construct a nonexpansive bijection that is not an isometry. 

Enumerate as $D_0,D_1,D_2,\dots$ the clusters $K_F$ with $|F|=1$, and enumerate
$E_0,E_1,\dots$ those with $|F|=2$; thus $|D_i|=3$ and $|E_j|=9$ for all $i,j$.
We define $f\colon\K{X_2}\to\K{X_2}$ as follows:

\begin{enumerate}
\item Define $f$ so that it maps $D_0\cup D_1\cup D_2$
bijectively onto $E_0$, and maps $D_i$ onto $D_{i-3}$ for every $i\geq 3$.
\item $f$ maps $E_j$ onto $E_{j+1}$ for every $j\geq0$.
\item $f$ is the identity in all remaining cases.
\end{enumerate}
The images in these three cases are pairwise disjoint and cover the whole hyperspace. Moreover, each cluster is mapped into a cluster. Thus $f$ is a nonexpansive bijection, but it is not an isometry. Indeed, if $A\in D_0$ and $B\in D_1$, then $d_H(A,B) = 2$, but $d_H(f(A), f(B)) = 1$, because $f(A), f(B)\in E_0$ are in the same cluster.
\end{proof}

\section{Graphs and graph metric spaces}\label{sec:graphs}

We use standard notation from graph theory. In particular, for a graph $G=(V,E)$, we write $u\sim v$ when the vertices $u$ and $v$ are adjacent, and $\deg(v)$ for the degree of a vertex $v$. We say that $G$ is connected if any two distinct vertices can be joined by a walk. The graph is \emph{locally finite} if $\deg(v)<\infty$ for every $v$, and
\emph{$k$-regular} if $\deg(v)=k$ for every $v$. A vertex of degree $k$ has a closed unit ball with $k+1$ vertices. We refer to \cite{Diestel,BondyMurty,ChartrandZhang} for additional notation and background on graph metric spaces.

\begin{definition}\label{def:graph}
Let $G=(V,E)$ be a connected graph. For $u,v\in V$ set
\[
d(u,v)=\min\{\,n\ :\ \text{there is a walk of length } n \text{ from } u \text{ to } v\,\}.
\]
Then $d$ is a metric on $V$, called the \emph{path metric}, and the pair $(V,d)$ is
called a \emph{graph metric space}. Note that $d(u,v)=1$ if and only if $u\sim v$, so
that the graph is recovered from the metric.  A walk of minimal length realizing
$d(u,v)$ is automatically a path, and we call it a \emph{geodesic} from $u$ to $v$.

If each edge $e\in E$ carries a weight $\ell(e)>0$ (a weighted graph), the
infimum of the sums of edge weights over walks joining two vertices defines
the \emph{weighted path distance}. Whenever distinct vertices have positive
distance, this is a metric, called the \emph{weighted path metric}, and
$(V,d)$ is then called a \emph{weighted graph metric space}. A common positive
lower bound for the edge weights is sufficient for this nondegeneracy. 
\end{definition}

Not every graph metric space is plastic, as the following examples illustrate.

\begin{example}[{cf.\ \cite[Example 3.2]{naimpally2006plasticity}}]\label{ex:weighted}
Let $X=\{z\in\Z:z<0\}\cup\{2n:n\in\NN\}$, regarded as a metric subspace of $\Z$, and let
$f\colon X\to X$ send each element to its predecessor in the natural ordering of
$X$. Then $f$ is a nonexpansive bijection and $d(f(0),f(2))=1<2=d(0,2)$, so $X$ is
not plastic. Note that $X$ is a weighted, 2-regular graph metric space.
\end{example}

In view of this example, we restrict our attention to graphs that are simple,
undirected, connected, and unweighted. 
In this setting, the path metric takes only integer values, and the space is discrete. In particular
\[
\K V=\{A\subset V: A \text{ is finite and nonempty}\}.
\]

\begin{lemma}\label{countable}
Let $G=(V,E)$ be a connected, locally finite graph. Then
every ball $\B v n$ is finite, and $V$ is countable. 
\end{lemma}

\begin{proof}
Fix $v\in V$. We argue by induction on $n$ that $\B v n$ is finite. This is clear for
$n=0$, and $\B v{n+1}=\bigcup_{u\in \B v n}\B u1$ is a finite union of finite sets.
Since $G$ is connected, $V=\bigcup_{n\in\NN}\B v n$ is a countable union of finite
sets. 
\end{proof}

\begin{lemma}\label{isometry-automorphism}
Let $G=(V,E)$ be a connected graph and let $f\colon V\to V$ be a bijection. Then $f$
is an isometry if and only if
\[
u\sim v \iff f(u)\sim f(v)
\qquad\text{for all } u,v\in V .
\]
\end{lemma}

\begin{proof}
If $f$ is an isometry, then $d(f(u),f(v))=d(u,v)$; in particular, $d(f(u),f(v))=1$ if and only if $d(u,v)=1$. Conversely, suppose $f$ preserves adjacency in both directions. Then $f$ sends paths to paths, and also, $f^{-1}$ sends paths to paths, from which we conclude that $f$ is an isometry.
\end{proof}

\begin{theorem}\label{Plasticity-regular-graph}
Let $G=(V,E)$ be a connected, locally finite, regular graph. Then $G$ is plastic.
\end{theorem}

\begin{proof}
Let $f\colon V\to V$ be a nonexpansive bijection and let $v\in V$. By
Lemma~\ref{general_fact_injectivity}, $f[\B v 1]\subset \B {f(v)} 1$. Since $G$ is
regular, $|\B v 1|=|\B {f(v)}1|<\infty$, and injectivity of $f$ therefore implies
that $f$ maps $\B v 1$ bijectively onto $\B {f(v)} 1$. In particular, $u\sim v$ if and
only if $f(u)\sim f(v)$. By Lemma~\ref{isometry-automorphism}, $f$ is an isometry,
and therefore $G$ is plastic.
\end{proof}

\begin{definition}\label{def:tree}
A \emph{tree} is a connected acyclic graph. For $k\ge2$ we write $\Td$ for the (infinite) \emph{regular tree of degree $k$}, that is, the unique tree in which every vertex has exactly $k$ neighbours. 
\end{definition}
In a tree, there is a unique path between any two vertices, and this path is the unique geodesic. We denote the unique path between $u$ and $v$ by $[u,v]$. Observe also that $T_2$ is isometric to $\Z$.

\begin{proposition}\label{trees-plastic}
Let $T$ be a tree and let $f\colon T\to T$ be a nonexpansive bijection. Then $f$
carries geodesics to geodesics. Consequently $f$ is an isometry, and every tree is
plastic.
\end{proposition}

\begin{proof}
Fix $u\in T$. We prove by induction on $n$ that for every $v\in\Sp u n$ the image
under $f$ of the geodesic from $u$ to $v$ is a geodesic from $f(u)$ to $f(v)$.

For $n=0$ there is nothing to prove. For $n=1$, the claim follows because $f$ is
nonexpansive and injective, so that $d(f(u),f(v))=1$.

Assume that the statement holds for $n$, and let $v\in\Sp u{n+1}$, with geodesic
$$u=v_0,v_1,\dots,v_n,v_{n+1}=v.$$ Since $T$ is a tree, it contains no cycles; therefore, $v_n\in\Sp u n$ and
$u=v_0,\dots,v_n$ is the geodesic from $u$ to $v_n$.  By the induction hypothesis
$f(v_0),\dots,f(v_n)$ is a geodesic from $f(u)$ to $f(v_n)$. Since $d(v_n,v_{n+1})=1$ and $f$ is nonexpansive and
injective, $d(f(v_n),f(v_{n+1}))=1$. Moreover, $f(v_{n+1})\notin\{f(v_0),\dots,f(v_n)\}$
because $f$ is injective. Hence
\[
f(v_0),f(v_1),\dots,f(v_n),f(v_{n+1})
\]
is a simple path in $T$. In a tree every simple path is the unique path between its
endpoints, hence a geodesic, and therefore $d(f(u),f(v))=n+1=d(u,v)$.
\end{proof}

\section{Plasticity of \texorpdfstring{$\K G$}{K(G)} for regular graphs}\label{sec:regular}

\subsection{Adjacent twins and minimum unit balls}\label{sec:twins-criterion}
If $G$ is finite, then both $G$ and $\K G$ are finite and hence plastic by Theorem~\ref{tb-plastic}. We therefore focus on infinite graph metric spaces. Finite subgraphs and finite metric subspaces will be used as auxiliary objects in the proofs. In particular, if $G$ is connected and regular, its degree is at least $2$.

\begin{definition}\label{def:adjacent-twins}
Let $(X,d)$ be a metric space. Two distinct points $u,v\in X$ are \emph{adjacent
twins} if $\B u1=\B v1$.
\end{definition}

For a graph $G$, the relation  $u\equiv v\Longleftrightarrow \B u1=\B v1$
is an equivalence relation. Its classes are
called \emph{adjacent-twin classes}, or simply \emph{twin classes}.
Each class is a clique and is finite when $G$ is locally finite. The name twins comes from the fact that adjacent twins are not detectable by the metric in the following sense:

\begin{lemma}\label{lem:twins-equiv}
Let $G$ be a connected graph and let $u\ne v$ be vertices. Then $\B u1=\B v1$ if and
only if $u\sim v$ and $d(u,z)=d(v,z)$ for every vertex $z\notin\{u,v\}$.
\end{lemma}

\begin{proof}
Suppose $\B u1=\B v1$. Then clearly $u\sim v$. Let $z\notin\{u,v\}$ and
let $u=v_0,v_1,\dots,v_n=z$ be a geodesic. If $v_1=v$, then $u=v_0,v_2,\dots,v_n=z$ would be a path from $u$ to $z$ of length $n-1$, because $v_2\in\B v1=\B u1$, a contradiction. Therefore, every geodesic $u=v_0,v_1,\dots,v_n=z$ gives rise to the path $v,v_1,\dots,v_n=z$ from $v$ to $z$. This shows that $d(v,z)\leq d(u,z)$. By symmetry, we obtain the reverse inequality and conclude that $d(u,z) = d(v,z)$.

Conversely, if $u\sim v$ and $d(u,z)=d(v,z)$ for all $z\notin\{u,v\}$, then for
$z\notin\{u,v\}$ we get $z\in\B u1\iff z\in\B v1$, while $u,v\in\B u1\cap\B v1$
because $u\sim v$. Hence $\B u1=\B v1$.
\end{proof}

For a metric space $(X,d)$ and $F\subset X$, we write
\[
N[F]=\bigcup_{a\in F}\B a1 ,
\]
so that $N[\emptyset]=\emptyset$. We also write $N[a]=N[\{a\}] =\B a1$.

We now estimate a lower bound for the sizes of unit balls in $\K G$. The first estimate uses
dominating sets in the finite subgraph induced by $A$.

\begin{definition}\label{def:dom}
Let $G$ be a connected graph and $A\subset G$. A set $D\subseteq A$
\emph{dominates} $A$ if $\B a1\cap D\ne\emptyset$ for every $a\in A$. Define
$$\gamma(A) = \min\{|D|: D \textrm{ dominates } A\}.$$
\end{definition}

\begin{lemma}\label{lem:dom-bound}
Let $G$ be a connected, locally finite graph and let $A\in\K G$ be finite. Write
$q=|N[A]\setminus A|$. Then
\[
|\B A1|\ \ge\ 2^{\,|A|+q-\gamma(A)} .
\]
\end{lemma}

\begin{proof}
Fix a dominating set $D\subseteq A$ with $|D|=\gamma(A)$. Suppose that $U\subset G$ satisfies
$D\subseteq U\subseteq N[A]$. Then $U$ is finite. Moreover, for every
$a\in A$, since $D$ dominates $A$, we have $\emptyset \neq D\cap\B a1 \subset U\cap\B a1$. By
Lemma~\ref{lem:hyperball}, this gives $U\in\B A1$. There are exactly $2^{\,|N[A]|-|D|}=2^{\,|A|+q-\gamma(A)}$ subsets $U$ such that $D\subseteq U\subseteq N[A]$.
\end{proof}

We use the following classical result on dominating sets.
\begin{proposition}[Ore's theorem]\label{ore}
Any undirected graph $G = (V,E)$ without isolated vertices\footnote{A vertex is isolated if it is not adjacent to any other vertex of $G$.} has a dominating set $D$ such that its complement $V\setminus D$ is also a dominating set. In particular, if $G$ has $n$ vertices, then it has a dominating set of at most $n/2$ vertices.
\end{proposition}
\begin{proof}
See \cite[Theorem 13.1.3]{Ore}.
\end{proof}

\begin{proposition}\label{prop:large-A}
Let $G$ be a connected, locally finite graph of minimum degree at least $k$, where $k\geq 1$. If $A\in \K G$ and $|A|>2^{\,k}$, then $|\B A1|>2^{k+1}-1$.
\end{proposition}

\begin{proof}
Let $n=|A|$ and $q=|N[A]\setminus A|$, and let
$$I = \{a\in A: \B a1\cap A = \set{a}\}$$
be the set of isolated vertices of the subgraph induced by $A$. Set $t=|I|$. Since $2^{\,k}\ge2k$ for every
$k\ge 1$, we have $n=|A|\ge 2^{\,k}+1\ge2k+1$.

\emph{Case $t=0$.} In this case, the subgraph induced by $A$ has no isolated vertices. Then, by Proposition~\ref{ore}, $\gamma(A)\le n/2$ and hence $n-\gamma(A)\ge\lceil n/2\rceil\ge k+1$, where the last
inequality follows from $n\ge2k+1$. Lemma~\ref{lem:dom-bound} gives
$|\B A1|\ge2^{\,n + q-\gamma(A)} \ge2^{\,n-\gamma(A)}\ge2^{k+1}$.

\emph{Case $0<t<n$.} Then $I\neq\emptyset$ and $I\neq A$. Each $a\in I$ has at least $k$ neighbours outside $A$, so
$q \ge k$. Moreover, $n-t\ge2$. Indeed, for any $a\in A\setminus I$, the definition of $I$ gives $|\B a1\cap A|\ge2$ and $\B a1\cap A\subset A\setminus I$.

Consider now the subgraph $G'$ induced by $A\setminus I$. If $G'$ has an isolated vertex $v$, then every neighbour of $v$ lies outside $A$ or belongs to $I$. The latter is impossible by the definition of $I$. Hence every neighbour of $v$ lies outside $A$, so $v\in I$, a contradiction. This shows that $G'$ has no isolated vertex. By Proposition~\ref{ore}, $G'$ admits a dominating set $D'$ with $|D'|\le(n-t)/2$. Therefore,
$I\cup D'$ dominates $A$, so $\gamma(A)\le |I| +|D'|\le t+(n-t)/2$. Hence,
$$n-\gamma(A)\ge n -(t+(n-t)/2) \ge(n-t)/2\ge1.$$
Applying Lemma~\ref{lem:dom-bound}, we obtain
$$|\B A1|\ge2^{\,q+n-\gamma(A)}\ge2^{k+1}.$$

\emph{Case $t=n$.} In this case, every vertex of $A$ is isolated in the subgraph induced by $A$, and each $a\in A$ has at least $k$ neighbours outside $A$, so
$q \ge k$. In this case, we prove the inequality by defining some elements in the unit ball directly. 
\begin{enumerate}
\item Every set $U$ such that $A\subseteq U\subseteq N[A]$ is an element of $\B A1$ and there are $2^q$ such sets.
\item Define $B=N[A]\setminus A$. For each $a\in A$ and each $T\subseteq B$ with
$T\cap\B a1\ne\emptyset$, the set $V(a,T)=(A\setminus\{a\})\cup T$ belongs to $\B A1$. For each $a$, we have at least $2^q - 2^{q-k}$ subsets $T$ such that $T\subseteq B$ and $T\cap\B a1\ne\emptyset$. Observe that if $T\neq T'$, then $V(a,T)\neq V(a,T')$ and $V(a,T)\neq V(b,T')$ if $a\neq b$, because $b\in V(a,T)$, but $b\notin V(b,T)$. In particular, there are at least $n\bigl(2^{q}-2^{\,q-k}\bigr)$ sets of the form $V(a,T)$.
\end{enumerate}

Moreover, every set $U$ in the first family contains $A$, whereas no set $V(a,T')$ does. Hence the two families are disjoint. Therefore,
$$
|\B A1|\ \ge\ 2^{q}+n\bigl(2^{q}-2^{\,q-k}\bigr) = 2^{q}\bigl(1 + n\bigl(1-2^{\,-k}\bigr)\bigr)
\ \ge\ 2^{\,k}+n\bigl(2^{\,k}-1\bigr).
$$
Since $n\ge2k+1$, we have
$$|\B A1|\ \ge 2^{\,k}+n\bigl(2^{\,k}-1\bigr)\ge 2^{\,k}+(2k+1)\bigl(2^{\,k}-1\bigr) = (k+1)2^{k+1}-2k-1.
$$
Thus,
$$|\B A1| - (2^{k+1} - 1)\geq (k+1)2^{k+1}-2k-1 - 2^{k+1} + 1 = k2^{k+1}-2k = k(2^{k+1}-2)>0
$$
for $k\geq 1$. This shows that
$$|\B A1| >2^{k+1} - 1.$$
\end{proof}

\begin{theorem}[Exact minimum and equality cases]\label{thm:min-hyperballs}
Let $G$ be a connected, locally finite graph, with minimum degree $k\ge1$. Then
\[
\min_{A\in\K G}|\B A1|=2^{k+1}-1.
\]
Moreover,  $|\B A1|=2^{k+1}-1$ if and only if $A$ is a nonempty
subset of a twin class whose vertices have the minimum degree $k$.
\end{theorem}
\begin{proof}
Observe that, for every $u\in G$, $|\B u 1|\geq k+1$ with equality if and only if $u$ is a vertex with the smallest degree $k$. Let $p=|N[A]|\ge k+1$. If $p=k+1$, then each $N[a]\subseteq N[A]$
has at least $k+1$ elements, so $N[a]=N[A]$ for every $a\in A$.
Thus all vertices of $A$ belong to one twin class of degree
$k$, and Lemma~\ref{lem:hyperball} gives
$\B A1=\K{N[a]}$, of size $2^{k+1}-1$.

Suppose $p\ge k+2$. If $|A|>2^{k}$, use
Proposition~\ref{prop:large-A}. Otherwise a subset of $N[A]$ fails
to be in the ball only if it misses some $N[a]$. Each such event
has at most $2^{p-k-1}$ elements. Therefore
\[
|\B A1|\ge2^p-|A|2^{p-k-1}
=2^{p-k-1}(2^{k+1}-|A|)\ge2\cdot2^{k}=2^{k+1}.
\]
Conversely, if $A$ is a nonempty subset of a twin class of degree $k$,
then all $N[a]$, $a\in A$, coincide and have $k+1$ elements, so the first case
gives $|\B A1|=2^{k+1}-1$. This proves both the lower bound and its equality
characterization.
A singleton at a vertex of minimum degree attains the bound.
\end{proof}

\subsection{The quotient and the general argument}\label{sec:quotient}
Theorem~\ref{thm:min-hyperballs} identifies the sets with smallest unit balls.
For a regular graph, they are exactly the nonempty subsets of twin classes.
We use these classes to describe the action of a nonexpansive bijection.

\begin{definition}\label{def_quotient}
Let $G$ be a connected, locally finite graph. Write $Q_G=G/{\equiv}$ for the quotient of $G$ by the relation $u\equiv v$ if and only if $u=v$ or $u$ and $v$ are adjacent twins. Two \emph{distinct} classes $C,D$ are adjacent in
$Q_G$ if some vertex of $C$ is adjacent to some vertex of $D$.
\end{definition}

Equality of the closed neighbourhoods within each class implies that all
vertices of two adjacent classes are mutually adjacent. Thus $Q_G$ is a
well-defined simple connected, locally finite graph. It need not be regular.
We denote its path metric by $d_Q$ and its closed neighbourhoods by $N_Q[C]$. Observe that every element of $Q_G$ is a finite subset of $G$, in particular, an element of $\K G$. The following proposition shows some connections of $Q_G$ and $\K G$:

\begin{proposition}\label{propertyQ}
Let $G$ be a connected, locally finite, $k$-regular graph and let $C,D\in Q_G$.
\begin{enumerate}
\item If $C\ne D$, then $d_Q(C,D)=d_G(c,d)$ for every $c\in C$ and $d\in D$.
In particular, $d_H(A,B)=d_Q(C,D)$ for nonempty $A\subseteq C$, $B\subseteq D$.
\item For $A\in\K C$, $c\in C$, and every integer $n\ge1$,
\[
\B A n=\K{\B c n},
\]
where the ball on the left is in $\K G$ and that on the right is in $G$.
\item 
\[
\B c1=\bigcup_{D\in N_Q[C]}D,
\qquad \sum_{D\in N_Q[C]}|D|=k+1.
\]
In particular, $|\B A1|=2^{k+1}-1$ for every $A\in\K C$.
\end{enumerate}
\end{proposition}
\begin{proof}
\begin{enumerate}
\item Fix $c\in C$ and $d\in D$. If $c = v_0,v_1,\ldots,v_n = d$ is a geodesic from $c$ to $d$, then, replacing each $v_i$ by its twin class and removing consecutive repetitions, we obtain a walk from $C$ to $D$. This walk contains a path of length at most $n$. In particular, $d_Q(C,D)\leq d_G(c,d)$. On the other hand, if $C = V_0,V_1,\ldots,V_n = D$ is a geodesic in $Q_G$, choose $v_0=c$, $v_n=d$ and $v_i\in V_i$ for $0<i<n$. We get a path of the same length from $c$ to $d$ in $G$. This shows that $d_G(c,d)\leq d_Q(C,D)$.

\item Let $B\in\K G$ and fix $c\in C$. Since every $a\in A$ belongs to $C$, Lemma~\ref{lem:twins-equiv} and $n\ge1$ give $\B a n=\B c n$ for every $a\in A$. By Lemma~\ref{lem:hyperball}, $d_H(A,B)\leq n$ if and only if $B\subseteq\B c n$ and $B\cap\B a n\ne\emptyset$ for every $a\in A$. Since $B$ is nonempty and all these balls coincide, the inclusion $B\subseteq\B c n$ already implies the intersection conditions. In particular, $d_H(A,B)\leq n$ if and only if $d_G(c,b)\leq n$ for every $b\in B$. This shows that $\B A n=\K{\B c n}$.

\item The classes meeting $\B c1$ are precisely $N_Q[C]$ and are contained
in that ball. They form a partition of its $k+1$ vertices, proving (3).
\end{enumerate}
\end{proof}

\begin{lemma}\label{lem:extension}
Let $G$ be a connected, locally finite, $k$-regular graph. Consider $Q=Q_G$.
Let $q:Q\to Q$ be a bijective isometry such that $|q(C)|=|C|$ for every $C\in Q$.
Then there is a bijective isometry $g:G\to G$ such that
$$q(C) = g[C]=\{g(c): c\in C\}$$
for every $C\in Q$.
\end{lemma}
\begin{proof}
For every class $C\in Q$, since $|q(C)|=|C|$, fix a bijection
$f_{C}:C\to q(C)$. Now, for every $v\in G=\bigcup_{C\in Q}C$,
let $C_v\in Q$ be the unique class such that $v\in C_v$. Define
$g(v)=f_{C_v}(v)$.

Then, $g$ is a bijection from $G$ onto $G$. Let us prove that $g$ is an isometry. Let $v,w\in G$. Consider
$C_v,C_w\in Q$ such that $v\in C_v$ and $w\in C_w$. Then
$g(v)=f_{C_v}(v)\in q(C_v)$ and
$g(w)=f_{C_w}(w)\in q(C_w)$.

If $v = w$, then $d(g(v), g(w)) = d(f_{C_v}(v), f_{C_v}(v)) = 0 = d(v,w)$. Assume now that $v\neq w$.
If $C_v=C_w$, then $v$ and $w$ are adjacent twins and since $f_{C_v}[C_v] = q(C_v)\in Q$ is a twin class, we conclude that $d(v,w) = 1 = d(g(v), g(w))$.

On the other hand, suppose $C_v\ne C_w$. Since $q$ is an isometry,
$$d_Q(q(C_v),q(C_w))=d_Q(C_v,C_w).$$
Now, by Proposition~\ref{propertyQ}(1),
$d_Q(q(C_v),q(C_w))=d(g(v),g(w))$ and
$d_Q(C_v,C_w)=d(v,w)$, where we conclude
$$d(g(v),g(w))=d(v,w).$$
\end{proof}

\begin{theorem}\label{automorphism_g}
Let $G$ be a connected, locally finite, $k$-regular graph and put $Q=Q_G$.
Let $F:\K G\to\K G$ be a nonexpansive bijection. Then there is a bijective isometry
$q:Q\to Q$ such that
$$F^{-1}[\K{C}]=\K{q(C)}$$
and $|q(C)|=|C|$ for every twin class $C\in Q$.
\end{theorem}
\begin{proof}
\textbf{The function $q$}: Fix $C\in Q$. Since $C$ is a set of twins,
$|\B{C}1|$ has the smallest possible cardinality
(Theorem~\ref{thm:min-hyperballs}). By
Lemma~\ref{general_fact_injectivity}, it follows that $A=F^{-1}(C)$
also has a unit ball of minimum size. Thus, there is a unique class
$D\in Q$ such that $A\in\K{D}$. Define $q(C)=D$.

\textbf{$q$ is injective}: Observe that, since $A\in\K{q(C)}$, by
Proposition~\ref{propertyQ}(2) and Lemma~\ref{general_fact_injectivity} we conclude that
$$F[\B A1]=F[\B{q(C)}1]=\B{C}1.$$
In particular, if $q(C)=q(D)$, then $\B{C}1=\B{D}1$.
By Proposition~\ref{propertyQ}(2), this means
$\K{\B c1}=\B{C}1 =\B{D}1 =\K{\B d1}$, for every $c\in C$ and $d\in D$. So $\B c1=\B d1$ and therefore $C=D$ because $c$ and $d$ would be twins for every $c\in C$ and $d\in D$.
This shows that the map is injective.

\textbf{$q$ preserves adjacency and class sizes}: Let $B\in\K{C}$
and $A=F^{-1}(B)$. As in the first paragraph, the minimum-ball criterion
shows that $A$ is contained in one twin class. Then, by
Lemma~\ref{general_fact_injectivity} and equality of the finite
cardinalities, $F[\B A1]=\B B1$ and therefore, by
Proposition~\ref{propertyQ}(2), since $\B B1=\B{C}1$ we conclude
$$F[\B A1]=\B B1=\B{C}1=F[\B{q(C)}1].$$
In particular, $\B A1=\B{q(C)}1$, which implies that the class containing
$A$ is $q(C)$. This proves that
$F^{-1}[\K{C}]\subseteq\K{q(C)}$. In particular, since $F$ is a bijection,
$$2^{|C|}-1\le2^{|q(C)|}-1$$
and therefore $|C|\leq |q(C)|$.

Let us prove that $q$ preserves adjacency. Let $C$ and $D$ be two
adjacent elements in $Q$. For any $A\in\K{C}$ and $B\in\K{D}$,
we have $A\in\B B1$. 
Since
$F[\B{q(D)}1]=\B{D}1=\B B1$, we conclude that
$F^{-1}(A)\in\B{q(D)}1$.
On the other hand, by the definition of $q$, we have $F^{-1}(A)\in\K{q(C)}$ and $F^{-1}(B)\in\K{q(D)}.$ 

Since $q(C)\ne q(D)$, Proposition~\ref{propertyQ}(1) implies that
$$d_Q(q(C), q(D)) = d_H(F^{-1}(A), F^{-1}(B)) = 1$$
where in the last statement we use the fact that $F^{-1}(A)\in\B{q(D)}1 =\B {F^{-1}(B)}1$. This proves that $q(C)$ and $q(D)$ are adjacent.
In particular,
$$q[N_Q[C]]\subseteq N_Q[q(C)].$$
Then, since $q$ is injective, Proposition~\ref{propertyQ}(3) gives
$$
k+1=\sum_{A\in N_Q[C]}|A|\le\sum_{A\in N_Q[C]}|q(A)|\le\sum_{A\in N_Q[q(C)]}|A|=k+1.
$$
This implies that
$$q[N_Q[C]]=N_Q[q(C)],$$
and moreover, $|q(D)|=|D|$ for every $D\in N_Q[C]$.
In particular, $|q(C)|=|C|$ for every class $C$  and therefore, the finite families $F^{-1}[\K{C}]$ and $\K{q(C)}$ have
the same cardinality, so we have
$$F^{-1}[\K{C}]=\K{q(C)}.$$

\textbf{$q$ is surjective}: Let $D\in Q$ and fix a class $C\in Q$.
Since $Q$ is connected, there is a path
$$q(C)=V_0,V_1,\ldots,V_n=D$$
that connects $q(C)$ and $D$. If $n=0$, then $q(C)=D$ and we are done.
Otherwise, since $q[N_Q[C]]=N_Q[q(C)]$ and $V_1$ is adjacent
to $q(C)$, there is a class $V_1'\in N_Q[C]$ such that
$q(V_1')=V_1$.
Again, if $n\ge2$, since
$$q[N_Q[V_1']]=N_Q[q(V_1')]=N_Q[V_1],$$
there exists $V_2'\in N_Q[V_1']$ such that $q(V_2')=V_2$.
Proceeding inductively, we obtain a class $V_n'\in Q$ such that $q(V_n')=V_n = D$, which proves that $q$ is surjective.

\textbf{$q$ is an isometry}: It remains to check that adjacency is also
preserved by the inverse. If $q(C)$ and $q(D)$ are adjacent, then
$q(D)\in N_Q[q(C)]=q[N_Q[C]]$. By injectivity, $D\in N_Q[C]$ and as the two classes are distinct, they are adjacent. Hence $q$ is a
bijection preserving adjacency in both directions, and
Lemma~\ref{isometry-automorphism} shows that $q$ is an isometry.
\end{proof}

\begin{theorem}\label{every-regular-graph-plastic}
Let $G$ be a connected, locally finite, $k$-regular graph. Then $\K G$ is plastic.
\end{theorem}
\begin{proof}
Let $F:\K G\to\K G$ be a nonexpansive bijection. By
Theorem~\ref{automorphism_g}, there exists an isometry $q:Q_G\to Q_G$ such that
$$F^{-1}[\K{C}]=\K{q(C)}$$
and $|q(C)|=|C|$ for every class $C\in Q_G$.
Now, by Lemma~\ref{lem:extension}, there is an isometry $g:G\to G$ such that
$q(C)=g[C]$. Define $g^*(A)=\{g(a):a\in A\}$.
Since $F[\K{C}]=\K{q^{-1}(C)}$, we have
$$(g^*\circ F)[\K{C}]
=g^*[\K{q^{-1}(C)}]=\K{C}.$$

Define $H=g^*\circ F:\K G\to\K G$ which is a nonexpansive bijection, because $g^*$ is an isometry. Let $C\in Q_G$, $A\in\K{C}$
and let $n\ge1$ be an integer. Then $A$ and $H(A)$ belong to $\K{C}$
and therefore, by Proposition~\ref{propertyQ}(2),
$\B A n=\B{H(A)}n=\K{\B c n}$, for every $c\in C$. Fix any $c\in C$.
Since $H$ is a nonexpansive injection, by
Lemma~\ref{general_fact_injectivity} and the finiteness of these balls
we conclude that
$$H[\K{\B c n}]=H[\B A n]=\B{H(A)}n=\K{\B c n}.$$
In particular, the restriction of $H$ is a nonexpansive bijection from
$\K{\B c n}$ onto itself. Since $\K{\B c n}$ is finite, and therefore
compact, by Theorem~\ref{tb-plastic} it is plastic. This shows that the
restriction of $H$ to $\K{\B c n}$ is an isometry.
The result now follows from the fact that for every $A,B\in\K G$, there
exists $n$ such that $A,B\in\K{\B c n}$. Thus $H$ is an isometry, and so
is $F$, since $g^*$ is an isometry.
\end{proof}

\begin{corollary}\label{cor:normalized-balls}
Let $G$ be a connected, locally finite, regular graph. For every nonexpansive
bijection $F$ of $\K G$, there is a bijective isometry $g_0$ of $G$ such that
\begin{equation}\label{eq:normalized-balls}
F[\K{\B v n}]=\K{g_0[\B v n]}
\quad(v\in G,\ n\ge1).
\end{equation}
\end{corollary}
\begin{proof}
  Consider $H = g^*\circ F$ as in the proof of
  Theorem~\ref{every-regular-graph-plastic}. Then
  $$H[\K{\B v n}]=\K{\B v n}.$$

  Since $H=g^*\circ F$, it follows that
  $$F[\K{\B v n}]
  =(g^*)^{-1}[\K{\B v n}]
  =\K{g^{-1}[\B v n]}.$$
  To conclude, take $g_0=g^{-1}$.
  \end{proof}
\subsection{Consequences}\label{sec:not-induced}
We now apply Theorem~\ref{automorphism_g} to describe when nonexpansive bijections of $\K G$ preserve singletons, and then consider regular trees.

We call two points $x\ne y$ of a metric space $X$ \emph{twins} if $d(x,z)=d(y,z)$ for
every $z\in X\setminus\{x,y\}$. It follows from Lemma~\ref{lem:twins-equiv} that, in a connected graph,
adjacent twins are exactly the twins at distance $1$.

\begin{lemma}\label{twins}
Let $X$ be a metric space containing twins $x\ne y$ with $d(x,y)=1$, and suppose that
$d(x,z)\ge1$ for every $z\in X\setminus\{x\}$. Then the map $F\colon\K X\to\K X$
defined by
\[
F(A)=
\begin{cases}
\{x,y\},&A=\{x\},\\
\{x\},&A=\{x,y\},\\
A,&\text{otherwise},
\end{cases}
\]
is an isometry of $\K X$ which does not map the set of singletons onto itself. In
particular, $F$ is not induced by any isometry of $X$.
\end{lemma}

\begin{proof}
It is not difficult to construct the inverse of $F$. In particular, $F$ is a bijection. Since $F$ moves only $\{x\}$ and $\{x,y\}$,
it suffices to prove that 
$$d_H(B,\{x\})=d_H(B,\{x,y\})$$
for every $B\in\K X$, with $B\notin \set{\set{x}, \set{x,y}}$.

The equality is immediate when $B=\{y\}$. Thus, we may assume that there exists $b_0\in B$ with $b_0\notin\{x,y\}$.

Since $d(x,B)=\inf_{b\in B}d(b,x)\le\sup_{b\in B}d(b,x)$, we have
$$d_H(B,\{x\})=\sup_{b\in B}d(b,x).$$
On the other hand, for each $b\in B$, we have
$d(b, \set{x,y}) = 0$ if $b\in \set{x,y}$ or $d(b, \set{x,y}) = d(b,x)$ if $b \notin\set{x,y}$. 

By hypothesis, $d(b_0,x)\ge1$, hence
$\sup_{b\in B\setminus\{x,y\}}d(b,x)\ge1$, while the possible contributions of
$x$ and $y$ to $\sup_{b\in B}d(b,x)$ are $d(x,x)=0$ and $d(y,x)=1$. Therefore
\[
\sup_{b\in B}d(b,\{x,y\})=\sup_{b\in B\setminus\{x,y\}}d(b,x)=\sup_{b\in B}d(b,x).
\]

Moreover,
\[
d(x,B)\le\sup_{b\in B}d(b,x),\qquad
d(y,B)\le d(y,b_0)=d(x,b_0)\le\sup_{b\in B}d(b,x),
\]
since $x$ and $y$ are twins. Hence,
$$d_H(B,\{x,y\})=\sup_{b\in B}d(b,x)=d_H(B,\{x\}).$$

\end{proof}

\begin{corollary}\label{cor:twins-main}
Let $G$ be an infinite connected, locally finite, regular graph. Then every
nonexpansive bijection of $\K G$ maps the set of singletons onto itself if and
only if $G$ has no adjacent twins. In this case, for every such bijection $F$
there is an isometry $f:G\to G$ such that $F(\{v\})=\{f(v)\}$ for every $v\in G$.
\end{corollary}
\begin{proof}
Let $F:\K G \to \K G$ be a nonexpansive bijection. By Theorem \ref{every-regular-graph-plastic}, we have that $F$ is in fact an isometry.
Suppose that $G$ has no adjacent twins. Then every class $C$ in $Q_G$ is a singleton.
By Theorem~\ref{automorphism_g}, $F^{-1}$ permutes the families $\K{C}$, which are singletons. In particular, $F$ maps
singletons onto singletons. This implies that there exists a map $f:G\to G$ such that 
$$F(\{v\}) = \{f(v)\}$$
for every $v\in G$ and it is not difficult to see that this map $f$ is an isometry.

Conversely, suppose that $G$ has adjacent twins $u\ne v$. By
Lemma~\ref{lem:twins-equiv}, they are metric twins at distance $1$.
Lemma~\ref{twins} gives an isometry of $\K G$ which maps $\{u\}$ to
$\{u,v\}$ and therefore does not preserve singletons.
\end{proof}

\begin{corollary}\label{main-regular-trees}
Let $\Td$ be a regular tree of degree $k\ge2$. Then $\K\Td$ is plastic.
Moreover, every nonexpansive bijection $F:\K\Td\to\K\Td$ maps the set of
singletons onto itself, and there is an isometry $f$ of $\Td$ with
$F(\{v\})=\{f(v)\}$ for every vertex $v$.
\end{corollary}
\begin{proof}
The plasticity of $\K\Td$ follows from
Theorem~\ref{every-regular-graph-plastic}. The second part follows from Corollary \ref{cor:twins-main} and from the fact that in a tree, there are no adjacent twins.

\end{proof}

\section{Products with a compact connected factor}\label{sec:extras}
In this section, we prove that $\K{K\times G}$ is plastic whenever $K$ is
compact and connected and $G$ is an infinite connected, locally finite,
regular graph. 

In this section, fix a compact connected metric space $(K,d_K)$ and a
metric space $(L,d_L)$ such that
\begin{enumerate}
\item[(P1)] $d_L(u,v)\ge1$ for all distinct $u,v\in L$, and every bounded
subset of $L$ is finite.
\end{enumerate}
In particular, $L$ is countable and every compact subset of $L$ is finite.
Consider $K\times L$ with the supremum metric. We write $\pi_K$ and $\pi_L$
for the two canonical projections.

The following two lemmas are minor variants of
\cite[Lemmas 4.1 and 4.2]{hida2026plasticity}, and their proofs are analogous.
We include them for completeness. They allow us to track the connected
components under a nonexpansive bijection and reduce the problem to $\K L$.

\begin{lemma}\label{connected-components}
For each $F \in \K L$, the set
$S_F=\{\, C \in \K {K \times L} : \pi_L[C] = F \,\}$
is a compact clopen subset of $\K{K\times L}$, and it is a connected component of
$\K {K \times L}$.
\end{lemma}

\begin{proof}
Let $C \in S_F$ and $D \notin S_F$. Since $C$ and $D$ differ in at least one point of
the second coordinate, (P1) gives $d_H(C,D) \geq 1$. In particular, $S_F$ is clopen in
$\K {K \times L}$. Moreover, $S_F$ is a closed subset of $\K{K\times F}$, which is
compact by \cite[Theorem 0.8]{nadler1978hyperspaces}, so $S_F$ is compact.

To see that $S_F$ is connected, write $F= \{v_1, \ldots, v_m\}$ and consider the map
$\varphi\colon {\K K}^m \to \K {K \times L}$ given by
$\varphi((C_i)_{i=1}^m) = \bigcup_{i=1}^m C_i\times\{v_i\}$. The map is onto $S_F$ and is continuous because
$\varphi[\prod_{i=1}^m B(C_i, r)] \subset B(\bigcup_{i=1}^m C_i\times\{v_i\}, r)$ for
every $r>0$ and every $(C_i)_{i=1}^m$. Since ${\K K}$ is connected
\cite[Theorem 1.13]{nadler1978hyperspaces}, so is ${\K K}^m$, and therefore $S_F$ is
connected. Since $S_F$ is both connected and clopen, it is a connected component.
\end{proof}

\begin{lemma}\label{inequality_hausdorff}
Let $F_1, F_2 \in \K L$. Then $d_H(A,B)\geq d_H(F_1,F_2)$ for every $A \in S_{F_1}$
and $B \in S_{F_2}$.
\end{lemma}

\begin{proof}
By symmetry, we may assume $d_H(F_1,F_2)=\sup_{v\in F_1} d(v,F_2)$. Since $F_1$ and
$F_2$ are finite, there exist $v_1\in F_1$ and $v_2\in F_2$ with
$d_H(F_1,F_2)=d(v_1,F_2)=d(v_1,v_2)$. Since $A\in S_{F_1}$, there exists $a\in K$ with
$(a,v_1) \in A$, and for any $(b,v)\in B$ we have $v\in F_2$, whence
$d((a,v_1),(b,v)) \geq d(v_1,v)\geq d(v_1, v_2)$. Therefore
\[
d_H(F_1,F_2) = d(v_1,v_2) \leq d\big((a,v_1),B\big)
\leq \sup_{(x,v)\in A} d((x,v),B) \leq d_H(A,B). \qedhere
\]
\end{proof}

\begin{lemma}[Sierpi\'nski]\label{sierpinski}
If a compact connected metric space $K$ has a countable cover by pairwise disjoint nonempty
closed subsets, then the cover has exactly one member.
\end{lemma}

\begin{proof}
See \cite[Theorem 6.1.27]{engelking}.
\end{proof}

\begin{lemma}\label{component-map-general}
Let $f:\K{K\times L}\to\K{K\times L}$ be a nonexpansive bijection. Then there
is a nonexpansive bijection $h:\K L\to\K L$ such that
\[
f[S_F]=S_{h(F)}\qquad(F\in\K L).
\]
\end{lemma}

\begin{proof}
Let $F \in \K L$. Since $S_F$ is connected (Lemma~\ref{connected-components}) and $f$
is continuous, $f[S_F]$ is connected. Hence, by Lemma~\ref{connected-components} there exists
$h(F)\in\K L$ with $f[S_F] \subset S_{h(F)}$.

\medskip
\noindent\textbf{Claim.} The map $h \colon \K L \to \K L$ is a nonexpansive bijection.

\smallskip
\noindent\emph{Injectivity.} Let $S = \{ H \in \K L : f[S_H] \subset S_{h(F)} \}$, so
that $F \in S$ and $f[\bigcup_{H \in S} S_H] \subset S_{h(F)}$. Let
$A\in S_{h(F)}$ and choose $P\in\K L$ such that $f^{-1}(A) \in S_P$. Then
$A \in f[S_P] \subset S_{h(P)}$. In particular, $h(P)=h(F)$ and $P \in S$. Then
\[
S_{h(F)} \subset f\Bigl[\bigcup_{H \in S} S_H \Bigr] = \bigcup_{H \in S} f[S_H].
\]
Each $f[S_H]$ is compact, hence closed, and these sets are pairwise disjoint because
$f$ is injective. Since $\K L$ is countable by (P1), we have that $S\subset \K L$ is countable. Then Lemma~\ref{sierpinski} gives
$|S|=1$, so $S=\{F\}$ and therefore, $h$ is injective. Moreover, $f[S_F]=S_{h(F)}$.

\smallskip
\noindent\emph{Surjectivity.} Given $F\in\K L$ and $A\in S_F$, choose $P\in\K L$ such that
$f^{-1}(A) \in S_P$. Then $A \in f[S_P] = S_{h(P)}$ and $A\in S_F$, so $F = h(P)$.

\smallskip
\noindent\emph{Nonexpansiveness.} Let $F_1,F_2\in\K L$ and fix $x\in K$. Then
$f(\{x\} \times F_i) \in S_{h(F_i)}$, so Lemma~\ref{inequality_hausdorff} gives
\[
\begin{aligned}
d_H(F_1,F_2)&=d_H(\{x\}\times F_1,\{x\}\times F_2)\\
&\ge d_H\bigl(f(\{x\}\times F_1),f(\{x\}\times F_2)\bigr)\\
&\ge d_H\bigl(h(F_1),h(F_2)\bigr),
\end{aligned}
\]
which proves the claim.

\end{proof}

\begin{theorem}\label{main_product}
Let $K$ be a compact connected metric space and let $L$ satisfy \textup{(P1)}.
Suppose also that
\begin{enumerate}
\item[(P2)] for every nonexpansive bijection $h$ of $\K L$, there exist a
bijective isometry $g:L\to L$ and nonempty finite sets
$L_1\subseteq L_2\subseteq\cdots$, with $\bigcup_nL_n=L$, such that
\[
h[\K{L_n}]=\K{g[L_n]}\qquad(n\ge1).
\]
\end{enumerate}
Then $\K{K\times L}$ is plastic.
\end{theorem}
\begin{proof}
Let $f:\K{K\times L}\to\K{K\times L}$ be a nonexpansive bijection.
By Lemma~\ref{component-map-general}, there is a nonexpansive bijection
$h:\K L\to\K L$ such that $f[S_F]=S_{h(F)}$ for every $F\in\K L$.

\textbf{Normalization}: By (P2), choose an isometry $g$ and sets $(L_n)_n$
such that $h[\K{L_n}]=\K{g[L_n]}$.
Define $J:K\times L\to K\times L$ by $J(x,v)=(x,g^{-1}(v))$.
Then $J$ is an isometry, and so is its induced map $J_*(A)=J[A]$.
Let $\widetilde f=J_*\circ f: \K{K\times L}\to\K{K\times L}$. Then 
$$\widetilde f[S_F] = (J_*\circ f)[S_F] = J_*[S_{h(F)}] = S_{g^{-1}[h(F)]}
$$
Define
$\widetilde h(F)=g^{-1}[h(F)]$ and let $(g^{-1})_*(A)=g^{-1}[A]$ be the map induced by $g^{-1}$ on $\K L$. Then $\widetilde f[S_F]=S_{\widetilde h(F)}$ and
$$\widetilde h[\K{L_n}]=(g^{-1})_*[h[\K{L_n}]] = (g^{-1})_*[\K{g[L_n]}] = \K{L_n}.$$

Let us prove that $\widetilde f$ restricts to an isometry from $\K{K\times L_n}$ into itself.

\textbf{Invariance}: Fix $n\ge1$. Let $A\in\K{K\times L_n}$ and suppose
$A\in S_F$. Then $F=\pi_L[A]\subseteq L_n$, so
$\widetilde h(F)\subseteq L_n$. Since
$\widetilde f(A)\in S_{\widetilde h(F)}$, we conclude that
$\widetilde f(A)\in\K{K\times L_n}$.
This shows that $\widetilde f[\K{K\times L_n}]\subset \K{K\times L_n}$.

To prove surjectivity of this restriction, let $B\in\K{K\times L_n}$
and write $E=\pi_L[B]\subset L_n$. Since
$\widetilde h[\K{L_n}]=\K{L_n}$, there exists $F\in\K{L_n}$ such that
$\widetilde h(F)=E$. Moreover,
$\widetilde f[S_F]=S_{\widetilde h(F)} = S_E$, so there is $A\in S_F$ with $\widetilde f(A)=B$.
As $F\subseteq L_n$, this set $A$ belongs to $\K{K\times L_n}$.
Thus $\widetilde f$ restricts to a nonexpansive bijection of
$\K{K\times L_n}$ onto itself.

\textbf{Isometry}: Since $L_n$ is finite and $K$ is compact, $K\times L_n$
is compact. By Proposition~\ref{hyper-total-bound},
$\K{K\times L_n}$ is plastic, and therefore the restriction of
$\widetilde f$ to this space is an isometry.
Now let $A,B\in\K{K\times L}$. Their projections onto $L$ are finite, and
the sets $L_n$ are increasing with union $L$. Hence there is $n$ such that
$\pi_L[A]\cup\pi_L[B]\subseteq L_n$. It follows that
\[
d_H(A,B)=d_H(\widetilde f(A),\widetilde f(B))
       =d_H(f(A),f(B)),
\]
where the last equality holds because $J_*$ is an isometry.
This shows that $\widetilde f$ is an isometric bijection. Since $\widetilde f=J_*\circ f$ and $J_*$ is also an isometric bijection, it follows that $f$ is an isometry and therefore, $\K{K\times L}$ is plastic.
\end{proof}

\begin{corollary}\label{cor:product-regular}
Let $K$ be a compact connected metric space and let $G$ be an infinite connected,
locally finite, regular graph. Then $\K{K\times G}$ is plastic.
In particular, $\K{K\times T_k}$ is plastic for every $k\ge2$.
\end{corollary}
\begin{proof}
Condition (P1) follows from Lemma~\ref{countable}. For every nonexpansive
bijection $h$ of $\K G$, Corollary~\ref{cor:normalized-balls} gives an
isometry $g$ such that
$h[\K{\B v n}]=\K{\B{g(v)}n}=\K{g[\B v n]}$ for every integer $n\ge1$. For a fixed $v$, define $L_n = \B v n$ for every $n\geq 1$. This gives condition (P2), so
Theorem~\ref{main_product} applies.
\end{proof}

\section{Nonregular graphs}\label{sec:nonregular}
In this section, we consider nonregular graphs. We begin with an example showing that connectedness and local finiteness alone do not imply plasticity, even for the graph itself.

\begin{example}\label{ex:nonplastic}
Let $G$ have vertex set $\Z$ and edges $z\sim z+1$ for every $z\in\Z$, together
with $n\sim n+2$ for every $n\in\NN$ (with $0\in\NN$). Then $G$ is connected and
locally finite. Consider the shift $f(z)=z+1$. Then $f$ is a nonexpansive bijection. However, $d(-1,1)=2$
while $f(-1)=0\sim2=f(1)$, so $f$ is not an isometry and $G$ is not plastic. Observe that $G$ is a nonregular graph metric space containing cycles.
\end{example}

In this example, the minimum degree is $2$, and every negative vertex has
this degree. Thus there are infinitely many vertices of minimum degree.
The criterion in Corollary~\ref{cor:finite-min-degree} will require this set
to be finite.

We now give a sufficient condition for a nonregular graph and its
hyperspace to be plastic. To this end, we apply the theory of \emph{plastic pairs} of Kadets and Zavarzina
\cite{KZpairs} to obtain a criterion based on finite metric subspaces.
 A pair $(X,Y)$ of metric spaces is called
\emph{plastic} if every nonexpansive bijection from $X$ onto $Y$ is an
isometry.

\begin{definition}
Let $(X,d)$ be a finite metric space. Put
\[
\delta(X)=\sum_{\{a,b\}\subseteq X,\ a\ne b}d(a,b),
\]
where each unordered pair is counted once.
\end{definition}

For a finite graph $G$ with the path metric, $\delta(G)$ is the \emph{Wiener index}
$W(G)$.

\begin{lemma}\label{plastic-pairs}
Let $X$ and $Y$ be finite metric spaces and let $F\colon X\to Y$ be a nonexpansive
bijection with $\delta(X)\le\delta(Y)$. Then $F$ is an isometry.
\end{lemma}

\begin{proof}
See \cite[Theorem 2.1]{KZpairs}.
\end{proof}

\begin{theorem}\label{general_finite_minimum}
Let $(X,d)$ be a metric space in which every ball is finite. Put
$m=\min\{|\B x1|:x\in X\}$ and $X'=\{x\in X:|\B x1|=m\}$. If $X'$ is finite, then
$X$ is plastic.
\end{theorem}

\begin{proof}
Let $f\colon X\to X$ be a nonexpansive bijection. If $f(x)\in X'$ then
$|\B x1|\le|\B{f(x)}1|=m$ by Lemma~\ref{general_fact_injectivity}, so $x\in X'$ by
minimality. Then $f^{-1}[X']\subseteq X'$ and hence $X'\subseteq f[X']$. As $X'$
is finite and $f$ is injective, $f[X']=X'$, so $f$ permutes the finite set $X'$.
Choose $x\in X'$ and $N\ge0$ with $f^{N+1}(x)=x$.

Fix $r>0$. By Lemma~\ref{general_fact_injectivity},
$f[\B{f^{i}(x)}r]\subseteq\B{f^{i+1}(x)}r$ for every $i$, so
\[
|\B xr|\le|\B{f(x)}r|\le\cdots\le|\B{f^{N+1}(x)}r|=|\B xr|.
\]
In particular, all these cardinalities are equal. By Lemma~\ref{general_fact_injectivity}, $f$
restricts to a nonexpansive bijection of $\B{f^{i}(x)}r$ onto $\B{f^{i+1}(x)}r$, so
$\delta(\B{f^{i+1}(x)}r)\le\delta(\B{f^{i}(x)}r)$. Using again the fact that $f^{N+1}(x)=x$, we conclude that $\delta(\B xr)=\delta(\B{f(x)}r)$. Thus $|\B xr|=|\B{f(x)}r|$ and
$\delta(\B xr)=\delta(\B{f(x)}r)$ for every $r$. By Lemma~\ref{plastic-pairs}, the restriction of $f$ to $\B xr$ is an isometry onto $\B{f(x)}r$ for every $r>0$. Since any two points of $X$ belong to a common ball $\B xr$, it follows that $f$ is an isometry.
\end{proof}

\begin{corollary}\label{cor:finite-min-degree}
Let $G$ be an infinite connected, locally finite graph, $k$ its minimum degree and
$G'=\{v:\deg(v)=k\}$. If $G'$ is finite, then both $G$ and $\K G$ are plastic.
\end{corollary}

\begin{proof}
All balls of $G$ are finite by Lemma~\ref{countable}, and the set of vertices at
which $|\B v1|$ is minimal is exactly $G'$, so
Theorem~\ref{general_finite_minimum} applies to $G$.

All balls of $\K G$ are finite as well. Indeed, Lemma~\ref{countable} shows that finite-radius neighbourhoods of finite subsets of $G$ are finite, and Lemma~\ref{lem:hyperball} then gives the finiteness of the corresponding Hausdorff balls. Let $M$ be the family of sets minimizing $|\B A1|$ in $\K G$.
Since $G$ is infinite and connected, $k\ge1$, and
Theorem~\ref{thm:min-hyperballs} gives
\[
M=\bigcup_{\substack{C\in Q_G\\ C\subseteq G'}}\K C
\subseteq\K{G'}.
\]
As $G'$ is finite, so is $M$. Thus
Theorem~\ref{general_finite_minimum} applies to $\K G$ as well.
\end{proof}

In particular, $\K T$ is plastic for every infinite locally finite tree with
a nonempty finite set of leaves, since the leaves are exactly its vertices
of minimum degree.

The preceding criterion applies when there are only finitely many vertices
of minimum degree. If there are infinitely many such vertices, it does not
give a conclusion. Moreover, the weighted-neighbourhood identity used in
Theorem~\ref{automorphism_g} depends on regularity. Thus the results of this
section cover only certain classes of nonregular graphs, and the following
question remains open.

\begin{question}
Let $G$ be a plastic infinite connected, locally finite graph. Is $\K G$ plastic?
\end{question}

\section*{Acknowledgements}

The author thanks G.~Vituri\footnote{viturivituri@gmail.com} for drawing his attention, during the XXVI Brazilian
Topology Meeting, to the space of Definition~\ref{def:Xm} as a candidate answer to
Question~\ref{main-question}.

\section*{Declaration of generative AI and AI-assisted technologies}

A generative AI model was used as a research and writing aid during the preparation of this manuscript. In particular, the AI-assisted tool supported the discussion and exploration of preliminary ideas and results developed by the authors, occasionally offering useful insights and suggesting possible directions for further investigation. All mathematical arguments, proofs, interpretations, and final formulations were independently developed, reviewed, and approved by the author, who assume full responsibility for the content of this work.

\end{document}